\documentclass[12pt]{article}
\usepackage{amssymb}
\usepackage{mathrsfs}
\usepackage[hypertex]{hyperref}
\usepackage{amsfonts}
\usepackage{graphicx}
\usepackage{mathptmx}
\usepackage{latexsym,amsmath,amssymb,amsfonts,amsthm}

\newtheorem{theorem}{Theorem}[section]
\newtheorem{lemma}[theorem]{Lemma}

\newtheorem{proposition}[theorem]{Proposition}
\newtheorem{remark}[theorem]{Remark}

\begin{document}
\setcounter{page}{1}
\title{A new way to Dirichlet problems  for minimal surface systems in arbitrary dimensions and codimensions}
\author{Jing Mao}
\date{}
\protect\footnotetext{\!\!\!\!\!\!\!\!\!\!\!\!{ MSC 2010:
53C21;53C44;35K55}
 }
\maketitle ~~~\\[-15mm]
\begin{center}{\footnotesize  Department of Mathematics, Harbin Institute
of Technology (Weihai), Weihai, 264209, China \\
jiner120@163.com, jiner120@tom.com}
\end{center}

\begin{abstract}
 In this paper, by considering a special case of the spacelike
mean curvature flow investigated by Li and Salavessa \cite{gi}, we
get a condition for the existence of smooth solutions of the
Dirichlet problem for the minimal surface equation in arbitrary
codimension. We also show that our condition is sharper than Wang's
in \cite[Theorem 1.1]{m1} provided the hyperbolic angle $\theta$ of
the initial spacelike submanifold $M_{0}$ satisfies
$\max_{M_{0}}{\rm cosh}\theta>\sqrt{2}$.
\end{abstract}

\markright{\sl\hfill  J. Mao \hfill}

\section{Introduction}
\renewcommand{\thesection}{\arabic{section}}
\renewcommand{\theequation}{\thesection.\arabic{equation}}
\setcounter{equation}{0} \setcounter{maintheorem}{0}

Let $\Omega$ be a bounded $C^{2}$ domain in the Euclidean $n$-space
$\mathbb{R}^{n}$ and $\phi:\partial\Omega\rightarrow\mathbb{R}^{m}$
be a continuous map from the boundary of $\Omega$ to
$\mathbb{R}^{m}$. The Dirichlet problem for the minimal surface
system asks whether there exists a Lipschitz map
$f:\Omega\rightarrow\mathbb{R}^{m}$ such that the graph of $f$ is a
minimal submanifold in $\mathbb{R}^{n+m}$ and
$f|_{\partial\Omega}=\phi$. For $m=1$ and any mean convex domain
$\Omega$, Jenkins and Serrin \cite{hj} proved the existence of the
solutions for this Dirichlet problem and the smoothness of all of
the solutions. The Dirichlet problem is well understood owing to the
pioneering works of Jenkins and Serrin \cite{hj}, De Giorgi
\cite{ed}, and Moser \cite{j}. However, they treated the Dirichlet
problem for just hypersurfaces. For surfaces with higher
codimension, very little is known. \emph{Is the high codimensional
Dirichlet problem solvable, or under what kind of assumptions could
one obtain the existence of solutions of this problem?} Lawson and
Osserman \cite{hbr} gave some nice examples to show how important
 the boundary data is for the solvability of high codimensional
Dirichlet problems. If the minimal submanifold is additionally
required to be Lagrangian, the minimal surface system becomes a
fully nonlinear scalar equation
\begin{eqnarray*}
\mathrm{Im}\left(\det(I+\sqrt{-1}D^{2}f)\right)=0,
\end{eqnarray*}
where $I$ is the identity matrix and
$D^{2}f=\left(\frac{\partial^{2}f}{\partial{x^ i}\partial{x^
j}}\right)$ is the Hessian matrix of $f$. Caffarelli, Nirenberg, and
Spruck \cite{jll} solved this Dirichlet problem with the prescribed
boundary value of $f$. For $n=2$ and any convex planar domain, the
existence of solutions of this problem was proved by Rad\'{o} in
\cite{tr} (see also \cite{hbr}). For $C^{2,\alpha}$ small Dirichlet
boundary data of finite codimension, Smale \cite{n} used the method
of linearization to prove the solvability successfully.

 Recently, under some assumption about the boundary data,
Wang \cite{m1} proved the existence of smooth solutions of the
Dirichlet problem for minimal surface systems in arbitrary
dimensions and codimensions by using his results in \cite{m2,m3} on
high codimensional mean curvature flow of submanifolds.
Surprisingly, by considering a special case of the spacelike mean
curvature flow (MCF for short) considered in \cite{gi}, we can prove
the following result.
\begin{theorem} \label{theorem1}
Let $\Omega$ be a bounded and closed $C^{2}$ convex domain in
$\mathbb{R}^{n}$ ($n\geq2$) with diameter $\delta$. If
$\psi:\Omega\rightarrow\mathbb{R}^{m}$ satisfies
\begin{eqnarray} \label{1.1}
4n\eta_{0}^{2}\delta\sup_{\Omega}\left|D^{2}\psi\right|+\sqrt{2}\sup_{\partial\Omega}\left|D\psi\right|<1,
\end{eqnarray}
then the Dirichlet problem for the minimal surface system is
solvable for $\psi|_{\partial\Omega}$ in smooth maps. Here
$\eta_{0}$ is a constant defined by (\ref{3.2}), depending only on
the spacelike graph of $\Omega$, and for $x\in\overline{\Omega}$,
\begin{eqnarray*}
|D\psi|(x):=\sup_{|v|=1}\left|D\psi(x)(v)\right|
\end{eqnarray*}
and
\begin{eqnarray*}
\left|D^ 2\psi\right|(x):=\sup_{|v|=1}\left|D^ 2\psi(x)(v,v)\right|
\end{eqnarray*}
are the norm and the squared norm of the differential
$D\psi(x):\mathbb{R}^{n}\rightarrow\mathbb{R}^{m}$.
\end{theorem}

The paper is organized as follows. We recall some useful facts about
spacelike MCF in \cite{gi}, and establish the relation between the
parametric and the non-parametric forms of the flow in Section 2. At
the end of Section 2, as in \cite{m1}, a boundary gradient estimate
is derived by using the initial map as a barrier surface. Theorem
\ref{theorem1} will be proved in the last section.

\section{Some useful facts}
\renewcommand{\thesection}{\arabic{section}}
\renewcommand{\theequation}{\thesection.\arabic{equation}}
\setcounter{equation}{0} \setcounter{maintheorem}{0}

 Assume the Riemannian manifold $(\Sigma_{1},g_{1})$ to be closed and of
dimension $n\geq2$, and the Riemannian manifold $(\Sigma_{2},g_{2})$
to be complete, of dimension $m\geq1$.  Let
$f:\Sigma_{1}\rightarrow\Sigma_{2}$ be a smooth map from
$(\Sigma_{1},g_{1})$ to $(\Sigma_{2},g_{2})$. Let
$\overline{M}=\Sigma_{1}\times\Sigma_{2}$ be a pseudo-Riemannian
manifold with the metric $\bar{g}=g_{1}-g_{2}$. Let $M$ be a
spacelike graph defined by
\begin{eqnarray*}
M=\Gamma_{f}=\left\{(p,f(p))|p\in{\Sigma_{1}}\right\},
\end{eqnarray*}
and denote by $g$ the induced metric on $M$. Clearly, if $f$ is a
constant map, $M$ is a slice. If we denote this spacelike immersion
by $F=id\times_{g}f$, then we say that the spacelike graph $M$
evolves along the MCF if
\begin{eqnarray} \label{2.1}
\left\{
\begin{array}{lll}
\frac{d}{d{t}}F(x,t)=H(x,t),
\quad  \forall{x}\in{M},~\forall{t}>0,\\
\\
F(\cdot,0)=F,  &\quad
\end{array}
\right.
\end{eqnarray}
where $H$ is the mean curvature vector of
$M_{t}=(M,F_{t}^{\ast}\bar{g})=F_{t}(M)$, and $id$ is the identity
map. The hyperbolic angle $\theta$ can be defined by (this
definition can also be seen in \cite{jal,gi1})
\begin{eqnarray} \label{2.2}
\mathrm{cosh}\theta=\frac{1}{\sqrt{\det(g_{1}-f^{\ast}g_{2})}},
\end{eqnarray}
which is used to measure the deviation from a spacelike submanifold
to a slice. Assume, in addition, that the Ricci curvature of
$\Sigma_{1}$ satisfies $Ricci_{1}(p)\geq0$, and the sectional
curvatures of $\Sigma_{1}$ and $\Sigma_{2}$ satisfy
$K_{1}(p)\geq{K_{2}(q)}$, for any $p\in\Sigma_{1}$,
$q\in\Sigma_{2}$. Besides, the curvature tensor $R_{2}$ of
$\Sigma_{2}$ and all its covariant derivatives are bounded. By
Theorem 1.1, Propositions 5.1, 5.2 and 5.3 in \cite{gi}, we have the
following.
\begin{theorem}\label{theorem2}
Let $f$ be a smooth map from $\Sigma_{1}$ to $\Sigma_{2}$ such that
$F_{0}:M\rightarrow\overline{M}$ is a compact spacelike graph of
$f$. Then
 $\\$ (1) A unique smooth solution of (\ref{2.1}) with initial
 condition $F_{0}$ a spacelike graphic submanifold exists in a
 maximal time interval $[0,T)$ for some $T>0$.
$\\$ (2) $\mathrm{cosh}\theta$ defined by (\ref{2.2}) has a finite
upper bound, and the evolving submanifold $M_{t}$ remains a
spacelike graph of a map $f_{t}:\Sigma_{1}\rightarrow\Sigma_{2}$
whenever the flow (\ref{2.1}) exists.
  $\\$(3) $\|B\|$, $\|H\|$, $\|\nabla^{k}B\|$, and $\|\nabla^{k}H\|$, for all
$k$, are uniformly bounded.
 $\\$ (4) The spacelike MCF (\ref{2.1}) exists for all the time.
\end{theorem}

Now, we would like to explain the connection between the spacelike
MCF (\ref{2.1}) and the high dimensional Dirichlet problem. Let
$\Omega\subset\mathbb{R}^{n}$ be a closed and bounded domain, and
$\psi:\Omega\rightarrow\mathbb{R}^{m}$ be a vector-valued function.
Then the graph of $\psi$ can be seen as the spacelike embedding
$id\times\psi:\Omega\rightarrow\mathbb{R}^{n}\times\mathbb{R}^{m}=\mathbb{R}^{n+m}$
with the pseudo-Riemannian metric
$\bar{g}=g_{1}-g_{2}=ds^{2}_{1}-ds^{2}_{2}$, where $ds^{2}_{1}$ and
$ds^{2}_{2}$ are the standard Euclidean metrics of $\mathbb{R}^{n}$
and $\mathbb{R}^{m}$, respectively. For the spacelike MCF
(\ref{2.1}), choosing $\Sigma_{1}=\Omega\subset\mathbb{R}^{n}$ and
$\Sigma_{2}=\mathbb{R}^{m}$, if we require
$F|_{\partial\Omega}=(id\times{\psi})|_{\partial\Omega}$, then the
immersed mapping $F_{t}$, with $F_{0}=F$, should be a smooth
parametric solution to the Dirichlet problem of the spacelike MCF,
that is,
\begin{eqnarray*}
\left\{
\begin{array}{lll}
\frac{dF}{dt}=H, \\
\\
F|_{\partial\Omega}=id\times{\psi}|_{\partial\Omega}.
\end{array}
\right.
\end{eqnarray*}
 In a local
coordinate system $\{x^{1},\ldots,x^{n}\}$ on $\Sigma_{1}=\Omega$,
the spacelike MCF is the solution
\begin{eqnarray*}
F=F^{A}(x^{1},\ldots,x^{n},t), \quad A=1,2,\ldots,n+m,
\end{eqnarray*}
to the following system of parabolic equations
\begin{eqnarray*}
\frac{\partial{F}}{\partial{t}}=\left(\sum_{i,j=1}^{n}g^{ij}\frac{\partial^{2}F}{\partial{x^{i}}\partial{x^{j}}}\right)^{\bot},
\end{eqnarray*}
where $g^{ij}=(g_{ij})^{-1}$ is the inverse of the induced metric
$g_{ij}=g\left(\frac{\partial{F}}{\partial{x^{i}}},\frac{\partial{F}}{\partial{x^{j}}}\right)$,
and $( \cdot )^{\top}$ and $( \cdot )^{\bot}$ denote the tangent and
the normal parts of a vector in $\mathbb{R}^{n+m}$, respectively.
\emph{The Einstein summation convention that repeated indices are
summed over is adopted in the rest of the paper}. As in the proof of
\cite[Lemma 2.1]{m1}, we can easily prove the following lemma.
\begin{lemma}  \label{lemmanormal}
We have
\begin{eqnarray*}
\Delta{F}=\frac{1}{\sqrt{G}}\frac{\partial}{\partial{x^{i}}}\left(\sqrt{G}g^{ij}\frac{\partial{F}}{\partial{x^{j}}}\right)
=\left(g^{ij}\frac{\partial^{2}F}{\partial{x^{i}}\partial{x^{j}}}\right)^{\bot},
\end{eqnarray*}
where $G=\det(g_{ij})$.
\end{lemma}

Lemma \ref{lemmanormal} tells us that $\Delta{F}$ is always in the
normal direction, which implies that
$g\left(\frac{\partial{F}}{\partial{x^{i}}},\Delta{F}\right)=0$ for
$1\leq{i}\leq{n}$.

Similar to \cite[Proposition 2.2]{m1}, we can also derive a relation
between parametric and non-parametric solutions to the spacelike MCF
equation.
\begin{proposition}
Suppose that $F$ is a solution to the Dirichlet problem for
spacelike MCF (\ref{2.1}) and that each $F(\Omega,t)$ can be written
as a graph over $\Omega\subset\mathbb{R}^{n}$. Then there exists a
family of diffeomorphisms $r_{t}$ of $\Omega$ such that
$\widetilde{F}_t = F_{t}\circ{r_{t}}$ is of the form
\begin{eqnarray*}
\widetilde{F}(x^{1},\ldots,x^{n})=(x^{1},\ldots,x^{n},f^{1},\ldots,f^{m})
\end{eqnarray*}
and
$f=(f^{1},\ldots,f^{m}):\Omega\times[0,T)\rightarrow\mathbb{R}^{m}$
satisfies
\begin{eqnarray} \label{2.4}
\left\{
\begin{array}{lll}
\frac{\partial{f^{\alpha}}}{\partial{t}}=g^{ij}\frac{\partial^{2}f^{\alpha}}{\partial{x^{i}}\partial{x^{j}}},
\quad \alpha=1,\ldots,m,\\
\\
f|_{\partial\Omega}=\psi|_{\partial\Omega},  &\quad
\end{array}
\right.
\end{eqnarray}
where
 \begin{eqnarray} \label{2.5}
 g^{ij}=(g_{ij})^{-1} \quad \mathrm{and} \quad
 g_{ij}=\delta_{ij}-\sum_{\beta=1}^{m}\frac{\partial{f^{\beta}}}{\partial{x^{i}}}\cdot\frac{\partial{f^{\beta}}}{\partial{x^{j}}}.
 \end{eqnarray}
 Conversely, if
 $f=(f^{1},\ldots,f^{m}):\Omega\times[0,T)\rightarrow\mathbb{R}^{m}$
 satisfies (\ref{2.4}), then $\widetilde{F}=I\times{f}$ is a
 solution to
 \begin{eqnarray*}
 \left(\frac{\partial}{\partial{t}}\widetilde{F}(x,t)\right)^{\bot}=\widetilde{H}(x,t).
 \end{eqnarray*}
\end{proposition}

By applying the maximum principle for scalar parabolic equations
(see, for instance, \cite{mh}) to the second-order parabolic
equation (\ref{2.4}), we have the following.
\begin{proposition}
Let
$f=(f^{1},\ldots,f^{m}):\Omega\times[0,T)\rightarrow\mathbb{R}^{m}$
be a solution to equation (\ref{2.4}). If
$\sup_{\Omega\times[0,T)}|Df|$ is bounded, then
\begin{eqnarray*}
\sup\limits_{\Omega\times[0,T)}f^{\alpha}\leq\sup\limits_{\Omega}\psi^{\alpha},
\end{eqnarray*}
with $\psi=(\psi^{1},\ldots,\psi^{m})$ the initial map given in
equation (\ref{2.4}).
\end{proposition}

By using the initial data $\psi:\Omega\rightarrow\mathbb{R}^{m}$ as
a barrier surface, we can obtain the boundary gradient estimate as
follows.
\begin{theorem} \label{theorem3}
 Let $\Omega$ be a bounded $C^{2}$ convex domain in
$\mathbb{R}^{n}$ with diameter $\delta$. Suppose that the flow
(\ref{2.4}) exists smoothly on $\Omega\times[0,T)$. Then we have
\begin{eqnarray*}
|Df|<\frac{4n\delta}{1-\xi}\sup_{\Omega}|D^{2}\psi|+\sqrt{2}\sup_{\partial\Omega}|D\psi|,
\quad \mathrm{on} ~~ \partial\Omega\times[0,T),
\end{eqnarray*}
where $\xi=\sup_{\Omega\times[0,T)}|Df|^{2}$.
\end{theorem}

\begin{proof}
We use a method similar to that of the proof of \cite[Theorem
3.1]{m1}. Denote by $P$ the supporting $(n-1)$-dimensional
hyperplane at a boundary point $p$, and $d_{p}$ the distance
function to $P$. Let $f =(f^{1},\ldots,f^{m})$ be a solution of
equation (\ref{2.4}). Consider the function defined by
\begin{eqnarray*}
S(x^{1},\ldots,x^{n},t)=v\log(1+kd_{p})-\left(f^{\alpha}-\psi^{\alpha}\right)
\end{eqnarray*}
on $\mathbb{R}^{n}$ for each $\alpha=1,2,\ldots,m$, where $k,v>0$
are to be determined. The Laplace operator on the graph
$(\Gamma_{f},g=F_{t}^{\ast}\bar{g})$ is given by
$\Delta=g^{ij}\frac{\partial^{2}}{\partial{x^{i}}\partial{x^{j}}}$,
with $g_{ij}$ satisfying (\ref{2.5}). Clearly,
\begin{eqnarray*}
g^{ij}=\left(\delta_{ij}-f_{i}f_{j}\right)^{-1}=\delta_{ij}+\frac{f_{i}f_{j}}{1-|Df|^{2}},
\end{eqnarray*}
where
$f_{i}f_{j}=\sum_{\beta=1}^{m}\frac{\partial{f^{\beta}}}{\partial{x^{i}}}\frac{\partial{f^{\beta}}}{\partial{x^{j}}}$.
Therefore, the eigenvalues of $g^{ij}$ are between 1 and
$1/(1-\xi)$. By direct computation, we know that $S$ satisfies the
following evolution equation
\begin{eqnarray} \label{2.6}
\left(\frac{d}{dt}-\Delta\right)S
=\frac{vk}{1+kd_{p}}(-\Delta{d_{p}})+\frac{vk^{2}}{(1+kd_{p})^{2}}g^{ij}\frac{\partial
d_{p}}{\partial{x^{i}}}\frac{\partial
d_{p}}{\partial{x^{j}}}-\Delta\psi^{\alpha}.
\end{eqnarray}
Since $d_{p}$ is a linear function, $\Delta{d_{p}}=0$, then
(\ref{2.6}) is reduced to
 \begin{eqnarray}  \label{2.7}
\left(\frac{d}{dt}-\Delta\right)S
=\frac{vk^{2}}{(1+kd_{p})^{2}}g^{ij}\frac{\partial
d_{p}}{\partial{x^{i}}}\frac{\partial
d_{p}}{\partial{x^{j}}}-\Delta\psi^{\alpha}.
 \end{eqnarray}
  Since $|Dd_{p}|=1$,
$d_{p}(y)\leq|p-y|\leq\delta$ for any $y\in\Omega$, and the fact
that the eigenvalues of $g^{ij}$ are between 1 and $1/(1-\xi)$, we
have
\begin{eqnarray*}
\frac{vk^{2}}{(1+kd_{p})^{2}}g^{ij}\frac{\partial
d_{p}}{\partial{x^{i}}}\frac{\partial
d_{p}}{\partial{x^{j}}}\geq\frac{vk^{2}}{(1+kd_{p})^{2}}\geq\frac{vk^{2}}{(1+k\delta)^{2}},
\end{eqnarray*}
and
\begin{eqnarray*}
\Delta\psi^{\alpha}=\left|g^{ij}\frac{\partial^{2}\psi^{\alpha}}{\partial{x^{i}}\partial{x^{j}}}\right|\leq\frac{n}{1-\xi}|D^{2}\psi|.
\end{eqnarray*}
Hence, if
\begin{eqnarray*}
\frac{vk^{2}}{(1+k\delta)^{2}}\geq\frac{n}{1-\xi}\sup\limits_{\Omega}|D^{2}\psi|,
\end{eqnarray*}
then, together with (\ref{2.7}), we have
$(\frac{d}{dt}-\Delta)S\geq0$ on $[0,T)$. On the one hand, by
convexity, we have $S>0$ on the boundary $\partial\Omega$ of
$\Omega$ except for $S=0$ at $p$. On the other hand, $S\geq0$ on
$\Omega$ at $t=0$. By the strong maximum principle for second-order
parabolic partial differential equations, we have
\begin{eqnarray*}
S>0, \qquad \mathrm{on}~\Omega\times(0,T).
\end{eqnarray*}
The same conclusion can also be obtained for a new function
$S':=v\log(1+kd_{p})+\left(f^{\alpha}-\psi^{\alpha}\right)$. Hence,
we have that the normal derivatives satisfy
\begin{eqnarray*}
\left|\frac{\partial\left(f^{\alpha}-\psi^{\alpha}\right)}{\partial{n}}\right|(p)&\leq&\lim\limits_{d_{p}(x)\rightarrow0}
\frac{\left|f^{\alpha}-\psi^{\alpha}\right|}{d_{p}(x)}\\
&<&\lim\limits_{d_{p}(x)\rightarrow0}\frac{v\log[1+kd_{p}(x)]}{d_{p}(x)}=vk.
\end{eqnarray*}
So, by changing coordinates of $\mathbb{R}^{m}$, we may assume
$\partial{f^\alpha}/\partial{n}=0$ for all $\alpha$ except for
$\alpha=1$ such that the inequality
\begin{eqnarray} \label{2.8}
\left|\frac{\partial{f}}{\partial{n}}\right|<vk+\left|\frac{\partial\psi}{\partial{n}}\right|
\end{eqnarray}
holds.

The Dirichlet boundary condition implies
\begin{eqnarray} \label{2.9}
\left|D^{\partial\Omega}f\right|=\left|D^{\partial\Omega}\psi\right|,
\qquad \mathrm{on}~\partial\Omega,
\end{eqnarray}
where $|D^{\partial\Omega}f|$ is defined by
$|D^{\partial\Omega}f|:=\sup_{w}|Df(x)w|$ for $x\in\partial\Omega$
and $w$ being taken over all unit vectors tangent to
$\partial\Omega$. Combining (\ref{2.8}) and (\ref{2.9}), we have
\begin{eqnarray*}
|Df|<\sqrt{\left(vk+\left|\frac{\partial\psi}{\partial{n}}\right|\right)^2+|D^{\partial\Omega}\psi|^2}\leq\sqrt{2}|D\psi|+vk,
\quad \mathrm{on} ~~\partial\Omega.
\end{eqnarray*}
Now, it is not difficult to find out that if we want to prove our
assertion here, we actually need to minimize $vk$ under the
constraint
$\frac{vk^{2}}{(1+k\delta)^{2}}\geq\frac{n}{1-\xi}\sup\limits_{\Omega}|D^{2}\psi|$.
In fact, the minimum $(vk)_{\mathrm{min}}$ of the function $vk$ is
obtained when $k=\delta^{-1}$, and
$(vk)_{\mathrm{min}}=4n\delta(1-\xi)^{-1}\sup_{\Omega}|D^{2}\psi|$.
The theorem follows.
\end{proof}

\section{Proof of the main theorem}
\renewcommand{\thesection}{\arabic{section}}
\renewcommand{\theequation}{\thesection.\arabic{equation}}
\setcounter{equation}{0} \setcounter{maintheorem}{0}

Now, by applying the conclusions we recalled and derived in Section
2, we can prove Theorem \ref{theorem1} as follows.

\textbf{\emph{Proof of Theorem \ref{theorem1}}}. We divide the proof
of Theorem \ref{theorem1} into five steps.

Step 1. By the Schauder fixed-point theorem (see, for instance,
Theorem 8.1 on p. 199 of \cite{gml2} for a detailed description and
the proof of \emph{the Schauder fixed-point theorem}), the
solvability of the parabolic system (\ref{2.4}) can be reduced to
the estimates of the solution $(f^{\alpha})$ of the uniformly
parabolic system
\begin{eqnarray}  \label{3.1}
\left\{
\begin{array}{lll}
\frac{d{f^{\alpha}}}{d{t}}=\tilde{g}^{ij}\frac{\partial^{2}f^{\alpha}}{\partial{x^{i}}\partial{x^{j}}},
\quad \alpha=1,\ldots,m,\\
\\
f|_{\partial\Omega}=\psi|_{\partial\Omega},  &\quad
\end{array}
\right.
\end{eqnarray}
with the coefficients
\begin{eqnarray*}
\tilde{g}^{ij}=\left(\delta_{ij}-\sum\limits_{\beta=1}^{m}\frac{\partial{u^{\beta}}}{\partial{x^{i}}}\frac{\partial{u^{\beta}}}{\partial{x^{j}}}\right)^{-1}
=\delta_{ij}+\frac{\sum\limits_{\beta=1}^{m}\frac{\partial{u^{\beta}}}{\partial{x^{i}}}\frac{\partial{u^{\beta}}}{\partial{x^{j}}}}{1-|Du|^{2}}
\end{eqnarray*}
for any $u=(u^1,\ldots,u^m)$ with uniform $C^{1,\gamma}$ bound. The
property of being uniformly parabolic of (\ref{3.1}) is equivalent
to $|Du|<1$ for all time $t\in[0,T)$. Fortunately, $|Du|<1$ for
$0\leq{t}<T$ essentially corresponds to the fact that the evolving
submanifold $M_{t}$ is spacelike for  all the time $t\in[0,T)$,
which can be obtained directly from Theorem \ref{theorem2} (2). Now,
(\ref{3.1}) is a decoupled system of linear parabolic equations,
which is uniformly parabolic and whose required estimate follows
from linear theory for scalar equations. Therefore, we know that the
system (\ref{2.4}) has the solution on a finite time interval.

In fact, there is another way to show the short-time existence of
the solution of (\ref{2.4}). More precisely, by Theorem
\ref{theorem2} (1), we can also get the short-time existence, since
in our case, as explained before, we choose $\Sigma_{1}$ to be a
closed domain in $\mathbb{R}^{n}$ and $\Sigma_{2}=\mathbb{R}^{m}$,
which implies that the system (\ref{2.4}) is just a special case of
the spacelike MCF (\ref{2.1}) provided we additionally require
$F|_{\partial\Omega}=(id\times\psi)|_{\partial\Omega}$, i.e.,
$f|_{\partial\Omega}=\psi|_{\partial\Omega}$.

Step 2. Denote the graph of $f_{t}$ by $M_{t}$. We show that
$|Df_{t}|< 1$ holds under the assumption of Theorem \ref{theorem1}.
Similar to (3.4) and (3.5) in \cite{gi}, for each point
$p\in\Sigma_{1}$, we can choose an orthonormal basis for the tangent
space $T_{p}M$ and for the normal space $N_{p}M$ given as follows
\begin{eqnarray*}
e_{i}=\frac{1}{\sqrt{1-\sum_{\beta}\lambda_{i\beta}^{2}}}\left(a_{i}+\sum_{\beta}\lambda_{i\beta}a_{\beta}\right),
\quad i=1,\ldots,n,
\end{eqnarray*}
\begin{eqnarray*}
e_{\alpha}=\frac{1}{\sqrt{1-\sum_{j}\lambda_{j\alpha}^{2}}}\left(a_{\alpha}+\sum_{j}\lambda_{j\alpha}a_{j}\right),
\quad \alpha=n+1,\ldots,n+m,
\end{eqnarray*}
where $\{a_{i}\}_{i=1,\ldots,n}$ is a $g_{1}$-orthonormal basis of
$T_{p}\Sigma_{1}$ of eigenvectors of $f^{\ast}g_{2}$,
$\{a_{\alpha}\}_{\alpha=n+1,\ldots,n+m}$ is a $g_{2}$-orthonormal
basis of $T_{f(p)}\Sigma_{2}$, and $df=-\lambda_{i\alpha}$ with
$\lambda_{i\alpha}=\delta_{\alpha,n+i}\lambda_{i}$. Here
$\lambda_{i}$, $1\leq{i}\leq{n}$, are the eigenvalues of
$f^{\ast}g_{2}$. So, the spacelike condition on $M$ implies
$\lambda_{i}^{2}<1$ for each $1\leq{i}\leq{n}$. We list them
non-increasingly as
$\lambda_{1}^{2}\geq\lambda_{1}^{2}\geq\ldots\geq\lambda_{n}^{2}\geq0$.
By the classical Weyl's perturbation theorem, the ordering
eigenvalues $\lambda_{i}^{2}:\Sigma_{1}\rightarrow[0,1)$ is a
continuous and locally Lipschitz function. For each $p\in
\Sigma_{1}$, denote by $s=s(p)=\{1,2,\cdots,n\}$ the rank of $f$ at
the point $p$, which implies $\lambda_{s}^{2}>0$ and
$\lambda_{s+1}=\lambda_{s+2}=\cdots=\lambda_{n}=0$. Therefore, we
have $s\leq\min\{m,n\}$. In fact, after this setting, we have
$\lambda_{i\alpha}=\delta_{\alpha,n+i}=0$ if $i>s$, or $\alpha>n+s$.
Under the orthonormal basis $\{e_{1},\ldots,e_{n},\ldots,e_{n+m}\}$,
by (\ref{2.2})
 the hyperbolic angle $\theta$ satisfies
\begin{eqnarray*}
\mathrm{cosh}\theta=\frac{1}{\sqrt{\det(g_{1}-f^{\ast}g_{2})}}=\frac{1}{\sqrt{\prod^{n}_{i=1}(1-\lambda_{i}^{2})}}.
\end{eqnarray*}
Let
 \begin{eqnarray} \label{3.2}
\eta_{t}:=\max_{M_{t}}\mathrm{cosh}\theta.
\end{eqnarray}
 By \cite[Proposition 4.3]{gi}, we know that (\emph{in our case, $\Sigma_{1}=\Omega\subset\mathbb{R}^{n}$,
$\Sigma_{2}=\mathbb{R}^{m}$}) the evolution equation of
$\mathrm{cosh}\theta$ here should be
\begin{eqnarray*}
\frac{d}{dt}\ln(\mathrm{cosh}\theta)=\triangle{\ln(\mathrm{cosh}\theta)}-\left\{\|B\|^{2}-\sum_{k,i=1}^{n}\lambda_{i}^{2}\left(h_{ik}^{m+i}\right)^{2}-
2\sum_{k,i<j}\lambda_{i}\lambda_{j}h_{ik}^{m+j}h_{jk}^{m+i}\right\},
\end{eqnarray*}
where
$\|B\|^{2}=\sum_{i,j=1}^{n}\sum_{\alpha=n+1}^{n+m}(h_{ij}^{\alpha})^{2}$
is the squared norm of the second fundamental form. Hence, there
exists some nonnegative constant $\ell$ such that
 \begin{eqnarray*}
\frac{d}{dt}\ln(\mathrm{cosh}\theta)\leq\Delta{\ln(\mathrm{cosh}\theta)}-\ell\|B\|^{2}\leq\Delta{\ln(\mathrm{cosh}\theta)}.
\end{eqnarray*}
Then by the maximum principle for parabolic equations, we can obtain
\begin{eqnarray*}
\eta_{t}\leq\eta_{0}=\max_{M_{0}}\mathrm{cosh}\theta
\end{eqnarray*}
for $0<t\leq{T}\leq\infty$, which implies $\lambda_{i}^{2}(t)<1$ and
\begin{eqnarray*}
1-\lambda_{i}^{2}(t)\geq\prod^{n}_{i=1}\left(1-\lambda_{i}^{2}(t)\right)\geq\frac{1}{\eta_{t}^{2}}\geq\frac{1}{\eta_{0}^{2}}
\end{eqnarray*}
for any $0<t\leq{T}\leq\infty$. On the other hand, if we assume
\begin{eqnarray*}
4n\eta_{0}^{2}\delta\sup_{\Omega}|D^{2}\psi|+\sqrt{2}\sup_{\partial\Omega}|D\psi|<1,
\end{eqnarray*}
then by integrating along a path in $\Omega$, we have
$\sup_{\Omega}|Df_{0}|=\sup_{\partial\Omega}|D\psi|<1$ initially.
Hence, by the above arguments, we have $|Df_{t}|<1$ for any
$0\leq{t}\leq{T}\leq\infty$.

Step 3. By Theorem \ref{theorem2} (3) and (4), we know that the
norms of the second fundamental form and all of its derivatives are
bounded, and the flow (\ref{2.1}) exists for all time. Hence, the
solution to the spacelike MCF (\ref{2.4}) exists smoothly in
$[0,\infty)$.

Step 4. By \cite[Corollary 6.1]{gi}, we know that there exists a
time sequence $t_{n}\rightarrow\infty$ such that
$\sup_{\Sigma_{1}}\|H_{t_{n}}\|\rightarrow0$ when
$t_{n}\rightarrow\infty$. Since we also have a gradient bound (see
Theorem \ref{theorem2} (3)), we can extract a subsequence $t_{i}$
such that the graph $M_{t_{i}}$ converges to a Lipschitz graph with
$\sup\|H\|=0$ and $|Df|<1$, which implies that the limit submanifold
is a minimal spacelike submanifold.

Step 5. Interior regularity of the limit follows from \cite[Theorem
4.1]{m1}, since the singular values $\lambda_{i}$ of $Df$ satisfy
 $|\lambda_{i}\lambda_{j}|\leq1-1/\eta_{0}^{2}$ almost everywhere for any
$i\neq{j}$. Boundary regularity follows from \cite[Theorem
2.3]{hbr}. Our proof is finished. \hfill $\square$

\begin{remark} \rm{Clearly, our condition (\ref{1.1}) is better than
that in \cite[Theorem 1.1]{m1} provided $\eta_{0}=\max_{M_{0}}{\rm
cosh}\theta>\sqrt{2}$.}
\end{remark}

\section*{Acknowledgments}

  The author was supported by the starting-up research fund (Grant No. HIT(WH)201320) supplied
by Harbin Institute of Technology (Weihai), the project (Grant No.
HIT.NSRIF.2015101) supported by Natural Scientific Research
Innovation Foundation in Harbin Institute of Technology, and the NSF
of China (Grant No. 11401131). The author would like to thank the
anonymous referee for his or her careful reading and valuable
comments such that the article appears as its present version.

 \end{document}